\documentclass[english]{amsart}
\usepackage[all,cmtip]{xy}

\usepackage{amsmath,amssymb,amscd,amsfonts}

\usepackage{amstext}
\usepackage{amsmath}
\usepackage{amsfonts}
\usepackage{latexsym}
\usepackage{ifthen}
\usepackage{xypic}
\usepackage[all,cmtip]{xy}
\xyoption{all}
\newcommand{\p}{{\rm p^{r}}}
\newcommand{\codim}{{\rm codim}}
\newcommand{\Pic}{{\rm Pic}}

\newcommand\loc{{\rm loc}}

\def\rank{\mathop{\rm rank}\nolimits}

\def\Pic{\mathop{\rm Pic}\nolimits}
\def\Det{\mathop{\rm det}\nolimits}

\def\Div{\mathop{\rm div}\nolimits}

\def\pr{\mathop{\rm pr}\nolimits}

\def\Sym{\mathop{\rm Sym}\nolimits}

\let\ep=\varepsilon

\newcounter{lemma}
\renewcommand{\thelemma}{\strut\kern-3pt\arabic{section}.\arabic{lemma}}
\newtheorem{lemma1}[lemma]{\setcounter{equation}{0}}
\let\saveref=\ref
\def\ref#1{\strut\kern3pt{\saveref{#1}}}
\def\eqref#1{({\saveref{#1}})}

\newenvironment{lemma}{\begin{lemma1}{\bf Lemma.}}{\end{lemma1}}

\newenvironment{theorem}{\begin{lemma1}{\bf Theorem.}}{\end{lemma1}}
\newenvironment{proposition}{\begin{lemma1}{\bf Proposition.}}{\end{lemma1}}
\newenvironment{corollary}{\begin{lemma1}{\bf Corollary.}}{\end{lemma1}}
\newenvironment{remark}{\begin{lemma1}{\bf Remark.}\rm}{\end{lemma1}}
\newenvironment{definition}{\begin{lemma1}{\bf Definition.}}{\end{lemma1}}
\newenvironment{conjecture}{\begin{lemma1}{\bf Conjecture.}}{\end{lemma1}}

\title[Manifolds with nef anticanonical bundles]{Albanese maps of projective manifolds with nef anticanonical bundles}

\begin{document}


\author{Junyan Cao}

\address{Junyan Cao, Sorbonne Universit\'e - Campus Pierre et Marie Curie \\
Institut de Math\'{e}matiques de Jussieu\\
4, Place Jussieu, Paris 75252, France }
\email{junyan.cao@imj-prg.fr}

\begin{abstract}
Let $X$ be a projective manifold such that the anticanonical bundle $-K_X$ is nef.
We prove that the Albanese map $p: X \rightarrow Y$ is locally trivial. In particular, $p$ is a submersion.

\medskip

\hspace{-15pt} R\'{e}sum\'{e}. Soit $X$ une vari\'{e}t\'{e} projective \`{a} fibr\'{e} anticanonique nef. On montre que l'application d'Albanese $p: X \rightarrow Y$
est localement triviale. En particulier, $p$ est lisse.
\end{abstract}

\maketitle

\section{Introduction}

Let $X$ be a compact K\"ahler manifold such that the anticanonical bundle $-K_X$ is nef, and 
let $p : X\rightarrow Y$ be the Albanese map. By the work of Q. Zhang \cite{Zha96} and M. P\v aun \cite{Pau12},
we know that $\pi$ is a fibration, i.e. $\pi$ is surjective and has connected fibres.
Conjecturally, the Albanese map has more regularities :

\begin{conjecture} \cite{DPS96} \label{conjecturealbanese}
Let $X$ be a compact K\"ahler manifold such that  $-K_X$ is nef, and 
let $p: X\rightarrow Y$ be the Albanese map.
Then $p$ is locally trivial, i.e., for any small open set $U\subset Y$, $p^{-1}(U)$ is biholomorphic to the product $U \times F$, where 
$F$ is the generic fibre of $p$. In particular, $p$ is a submersion.
\end{conjecture}

This conjecture has been proved under the 
stronger assumption that $T_X$ is nef, $-K_X$ is hermitian or the anticanonical bundle of the generic fibre is big 
\cite{CP91,DPS94,DPS96,CDP14,CH17a}. For the general case, \cite{LTZZ10} proved that $p$ is equidimensional and 
has reduced fibres.
In low dimension,  \cite{PS98} proved that the Albanese map is a submersion for $3$-dimensional projective manifolds. 

\medskip

The aim of this article is to prove the conjecture under the assumption that $X$ is projective :
\begin{theorem}\label{mainthm}
Let $X$ be a projective manifold with nef anti-canonical bundle and let $p: X\to Y$ be the Albanese map. 
Then $p$ is locally trivial, i.e., for any small open set $U\subset Y$, $p^{-1}(U)$ is biholomorphic to the product $U \times F$, where 
$F$ is the generic fibre of $p$.
\end{theorem}

\medskip

As an application of Theorem \ref{mainthm}, we can study the structure of the universal cover of 
projective manifolds with nef anticanonical bundles. Recalling that, for 
a compact K\"{a}hler manifold with hermitian semipositive anticanonical bundle $X$, 
\cite{DPS96,CDP14} proved that, the universal covering $\widetilde{X}$ admits a holomorphic and isometric splitting 
$$\widetilde{X} \simeq \mathbb{C}^q \times \prod Y_j \times \prod S_k \times \prod Z_l ,$$
where $Y_j$ are irreducible Calabi-Yau manifolds, $S_k$ are irreducible hyperk\"{a}hler manifolds,
and $Z_l$ are rationally connected manifolds  with irreducible holonomy. They expect a similar splitting result for 
compact K\"{a}hler manifolds with nef anticanonical bundles\footnote{Very recently, \cite{CH17b} proved the conjecture for projective manifolds with nef anticanonical bundles.
}:

\begin{conjecture}\label{newcon}
Let $X$ be a compact K\"{a}hler manifold with nef anticanonical bundle. Then the universal covering $\widetilde{X}$ of $X$ admits the following splitting
$$\widetilde{X} \simeq \mathbb{C}^q \times \prod Y_j \times \prod S_k \times Z  ,$$
where $Y_j$ are irreducible Calabi-Yau manifolds, $S_k$ are irreducible hyperk\"{a}hler manifolds,
and $Z$ is a rationally connected manifold.
\end{conjecture}

\noindent This conjecture was proved for $3$-dimensional projective manifolds \cite{BP04}. For an arbitrary compact K\"{a}hler manifold $X$ with nef anticanonical bundle, 
thanks to \cite{Cam95, Pau97, Pau12}, we know that 
the fundamental group $\pi_1 (X)$ of $X$ is almost abelian (cf. also Proposition \ref{isomalb}).
Together with Theorem \ref{mainthm}, we get the following partial result for Conjecture \ref{newcon}.

\begin{corollary}\label{structurere}
Let $X$ be a projective manifold with nef anticanonical bundle. Then the universal cover $\widetilde{X}$ of $X$ admits the following splitting
$$\widetilde{X} \simeq \mathbb{C}^r \times F .$$
Here $F$ is a compact simply connected projective manifold with nef anticanonical bundle, and $r = \sup h^{1,0} (\widehat{X})$ where the supremum is taken over
all finite \'{e}tale covers $\widehat{X} \rightarrow X$.
\end{corollary}

\medskip

Let us explain briefly the basic ideas of the proof of Theorem \ref{mainthm}.  
Like many works on the study of the manifolds with nef anticanonical bundles (cf. \cite{BC16,CH17a,CPZ03,CZ13,DPS93,Den17a,FG12,LTZZ10,Ou17,Pau12,Zha05} to quote only a few), 
the proof of Theorem \ref{mainthm} is based on the positivity of direct images.
More precisely, in the setting of Theorem \ref{mainthm}, 
let $L$ be a pseudo-effective line bundle on $X$
and let $A$ be an ample line bundle on $X$. In general, we don't know about the positivity of $p_\star (L +A)$. 
However, as $-K_{X/Y}$ is nef in our case, we can obtain the positivity of $p_\star (L +A)$ by using 
the following very elegant argument in \cite{Zha05}.

Fix a possibly singular metric $h_L$ such that $i\Theta_{h_L} (L) \geq 0$ in the sense of current
and let $m\in\mathbb{N}$ large enough such that $\mathcal{J} (h_L ^{\frac{1}{m}}) =\mathcal{O}_X$
\footnote{We refer to the paragraph
before Theorem \ref{extension} for the definition of $\mathcal{J} (h_L ^{\frac{1}{m}})$.}. 
We have 
\begin{equation}\label{keytrick}
L+A = m K_{X/Y} + (-m K_{X/Y} +A) +L . 
\end{equation}
As $-K_{X/Y}$ is nef, $(-m K_{X/Y} +A) $ is ample and can be equipped with a smooth metric $h_1$ with positive curvature. 
Therefore $h= h_1 +h_L$ defines a possibly singular metric on
\begin{equation}\label{introdecom}
\widetilde{L} := (-m K_{X/Y} +A) +L 
\end{equation}
with $i\Theta_h (\widetilde{L})\geq 0$ and $\mathcal{J} (h^{\frac{1}{m}}) =\mathcal{O}_X$. 
Then the powerful results on the positivity of direct images (cf.\cite{BP08,BP10,Kaw82,Kaw98, Kol85,Fuj16,PT14,Tsu10,Vie95} among many others) can be used to study the direct image
$$p_\star (m K_{X/Y} + \widetilde{L}) =p_\star (L+A) .$$ 
We refer to Proposition \ref{posdir} and Corollary \ref{usefulcor} for some more accurate statements.

\smallskip

Another main ingredient involved in the proof is inspired and very close to \cite[3.D]{DPS94} and \cite{CPZ03}.
Recalling that, under the assumption that $-K_{X/Y}$ is $p$-ample, \cite{DPS94} proved that $p_\star (- m K_{X/Y})$
is numerically flat for every $m\in\mathbb{N}$. Thanks to this numerically flatness, we can prove the local trivialness of the Albanese map
\cite{DPS94, CH17a}.
In the situation of Theorem \ref{mainthm}, as $- K_{X/Y}$ is not necessarily strictly positive along the fibres, 
we consider an arbitrary $p$-ample line bundle $L$ on $X$ to replace $-K_{X/Y}$.
By \cite{LTZZ10}, we can assume that $p_\star (m L)$ is locally free for every $m\in\mathbb{N}$.
By combining \cite[3.D]{DPS94} with the positivity of direct images discussed above, we can prove that,
$p_\star (m L')$ is numerically flat for every $m\in \mathbb{N}$, where $L':= \rank p_\star (L) \cdot L - p^\star \det p_\star (L)$.
The fibration $p$ is thus locally trivial by using a criteria proved in \cite{DPS94, CH17a}, cf. also Proposition \ref{isotri}.

\medskip

Here are the main steps of the proof of Theorem \ref{mainthm}. Firstly, 
using the positivity of direct images \cite{BP10}, the diagonal method of Viehweg
\cite[Thm 6.24]{Vie95} as well as the method of Zhang \cite{Zha05}, 
we prove in Proposition \ref{psf} that for any $p$-ample line bundle $A$ on $X$,
if $p_\star (A)$ is locally free, then $r A - p^\star\det p_\star (A)$ is pseudo-effective, where $r$ is the rank of $p_\star (A)$.
Secondly, after passing to some isogeny of the abelian variety $Y$, we can assume that $\frac{1}{r }\det p_\star (A)$ is a line bundle. 
By using an isogeny argument \cite[Lemma 3.21]{DPS94} and \cite{BP10}, we prove that 
$p_\star (A) \otimes  (-\frac{1}{r}\det p_\star (A))$ is numerically flat. 
Finally, we use the arguments in \cite{DPS94,CH17a} to conclude that $p$ is locally trivial.

\medskip

Our paper is organized as follows. In Section \ref{pre}, after recalling some basic notations and results about the positivity of line bundles and vector bundles,
we will review a criteria of the locally trivialness in \cite{CH17a}.
We will also gather some results about the positivity of direct images in \cite{BP08, BP10, PT14}.
In Section \ref{twolemme}, inspired by \cite[Section 3.D]{DPS94}, we will prove two important propositions which will be the key ingredients in the proof
of main theorem \ref{mainthm}. Both propositions imply in particular that the Albanese map is very rigid.
Finally, a complete proof of Theorem \ref{mainthm} and Corollary \ref{structurere} is provided in Section \ref{mainsect}.

\medskip

\noindent {\bf Acknowledgements.} 
We thank S. Boucksom, J.-P. Demailly, A. H\"{o}ring, S.S.Y. Lu and  M. Maculan for helpful discussion about the article.
We thank in particular Y. Deng and M. P\u{a}un for their numerous comments and suggestions about the text. 
We would like to thank also the anonymous referee for the constructive suggestions who helped us to improve substantially the quality of the work.
This work was partially supported by the Agence Nationale de la Recherche grant ``Convergence de Gromov-Hausdorff en g\'{e}om\'{e}trie k\"{a}hl\'{e}rienne"
(ANR-GRACK).

\section{Preparation}\label{pre}

We first recall some basic notations about the positivity of line bundles and vector bundles. We refer to \cite{Dem12, DPS94, Laz} for more details.

\begin{definition}\label{listdef}
 Let $X$ be a projective manifold.
 
 \begin{enumerate}
  \item We say that a holomorphic line bundle $L$ over $X$ is numerically effective, nef for short, if $L\cdot C \geq 0$ for every curve $C\subset X$.
 
 \noindent Thanks to \cite[Prop 6.10]{Dem12}, a line bundle $L$ to be nef is equivalent to say that  
 for every $\epsilon >0$,  there is a smooth hermitian metric $h_\epsilon$ on $L$ such that $i\Theta_{h_\epsilon} (L) \geq -\epsilon \omega$ where $\omega$ is a fixed
 K\"{a}hler metric on $X$. 
 
 \item We say that a holomorphic vector bundle $E$ over $X$ is nef, if $\mathcal{O}_{\mathbb{P} (E)} (1)$ is nef on $\mathbb{P} (E)$.
 
 \item We say that a holomorphic vector bundle $E$ over $X$ is numerically flat if both $E$ and its dual $E^\star$ are nef.
 
 \noindent It is easy to see that $E$ is numerically flat if and only if $c_1 (\det E)=0$ and $E$ is nef.

\item Let $p : X\rightarrow Y$ be a fibration between two projective manifolds and let $L$ be a line bundle on $X$. We say that $L$ is $p$-ample (resp. $p$-very ample), 
if there exists a 
line bundle $L_Y$ on $Y$ such that $L+ p^\star L_Y$ is ample (resp. very ample).
 \end{enumerate}
\end{definition}

The following two theorems about the numerically flat vector bundles will be useful for us.

\begin{theorem}\label{keyflat}
Let $X$ be a compact K\"{a}hler manifold and let $E$ be a numerically flat vector bundle on $X$. Then 

\begin{enumerate}
 \item \cite[Thm 1.18]{DPS94} $E$ admits a filtration
$$\{0\} = E_0 \subset E_1 \subset\cdots \subset E_p =E$$
by vector subbundles such that the quotients $E_k / E_{k-1}$ are hermitian flat.\\

\item \cite[Section 3]{Sim92} $E$ is a local system and the natural Gauss-Manin connection $D_E$ on $E$ is compatible with the natural flat connection on 
the quotients $E_k / E_{k-1}$, i.e., $D_E  (E_k ) \subset E_k \otimes \Omega_X ^1$ and the induced connection $D_E$ on $E_k / E_{k-1}$ coincides with
the hermitian flat connection on $E_k / E_{k-1}$ for every $k$.
\end{enumerate}
\end{theorem}

\begin{remark}
Recently, Y. Deng \cite[Chapter 6]{Den17b} gave an elegant and short proof of Theorem \ref{keyflat} $(2)$. 
In the case $X$ is a torus, we refer also to \cite[Lem 6.5, Cor 6.6]{Ver04} for a short proof of Theorem \ref{keyflat} $(2)$. 
\end{remark}

Theorem \ref{keyflat} implies the following criteria, which will be useful for the proof of Theorem \ref{mainthm}.

\begin{proposition}\cite[Prop 4.1]{CH17a}\label{isotri}
Let $p: X\rightarrow Y$ be a flat fibration between two compact K\"{a}hler manifolds and let $L$ be a $p$-very ample line bundle (cf. Definition \ref{listdef} (4)).
Set $E_m := p_\star ( m L)$. If $E_m$ is numerically flat for every $m \geq 1$, then $p$ is locally trivial. Moreover, let 
$\pi: \widetilde{Y} \rightarrow Y$ be the 
universal cover of $Y$. Then $\widetilde{X} := X \times_Y \widetilde{Y}$ admits the following splitting
$$\widetilde{X}  \simeq \widetilde{Y} \times F ,$$
where $F$ is the generic fibre of $p$.
\end{proposition}

For the reader's convenience, we give the proof of it.
\begin{proof}
As $L$ is $p$-very ample, we have a $p$-relative embedding and $L = j^\star \mathcal{O}_{\mathbb{P} ( E_1 )} (1)$. 
$$
\xymatrix{
X \ar[rd]_p \ar@{^{(}->}[rr]^j & &\mathbb{P} (E_{1}) \ar[ld]^f\\
& Y}
$$ 
For $m$ large enough, we have the exact sequence 
\begin{equation}\label{exactseq1pr}
0 \rightarrow f_\star (\mathcal{I}_X \otimes \mathcal{O} _{\mathbb{P} (E_1)} (m)) \rightarrow 
f_\star ( \mathcal{O} _{\mathbb{P} (E_1)} (m)) \rightarrow p_\star (m L) \rightarrow 0 . 
\end{equation}
As $E_1$ is numerically flat, Theorem \ref{keyflat} implies that $E_1$ is a local system. Let $D_{E_1}$ be the flat connection 
with respect to this local system. We assume that $n=\rank E_1$. 
Since $\widetilde{Y}$ is simply connected, $\pi^\star E_1$ is a trivial vector bundle on $\widetilde{Y}$ 
and we can take some flat sections (with respect to $D_{E_1}$) 
$$\{ e_1, e_2, \cdots, e_n \} \subset H^0 (\widetilde{Y}, \pi^\star E_1)$$ 
such that $\{ e_1, e_2, \cdots, e_n \} $ generates $\pi^\star E_1$.

\smallskip

Set $ F_m := f_\star (\mathcal{I}_X \otimes \mathcal{O} _{\mathbb{P} (E_1)} (m))$. Since both 
$f_\star ( \mathcal{O} _{\mathbb{P} (E_1)} (m)) =\Sym^m E_1$ and $p_\star (m L)$ are numerically flat by assumption, $F_m$ is also numerically flat.
Then $F_m$ is a local system and let $D_{F_m}$ be the flat connection on it. 
Then $\pi^\star F_m$ is a trivial vector bundle, and we can take some flat sections (with respect to $D_{F_m}$) 
$$\{ s_1, s_2, \cdots, s_t \} \subset H^0 (\widetilde{Y} , \pi^\star F_m)$$ 
such that $\{ s_1, s_2, \cdots, s_t \} $ generates $\pi^\star F_m$, where $t$ is the rank of $F_m$.

Let $\varphi: \pi^\star f_\star (\mathcal{I}_X \otimes \mathcal{O} _{\mathbb{P} (E_1)} (m)) \rightarrow 
\pi^\star f_\star ( \mathcal{O} _{\mathbb{P} (E_1)} (m))=\pi^\star \Sym^m E_1$ be the inclusion induced by \eqref{exactseq1pr}.
Let $D_{\Sym^m E_1}$ be the flat connection on $\Sym^m \pi^\star E_1$ induced by $D_{E_1}$.
Thanks to \cite[Lemma 4.3.3]{Cao13}, we know that for every $i$, $\varphi (s_i)$ is flat with respect the connection $D_{\Sym^m E_1}$.
In particular, for every $i$, we can find constants $a_{i, \alpha}$ such that 
$$\varphi (s_i)= \sum_{\alpha= (\alpha_1 ,\cdots, \alpha_n), |\alpha|=m} a_{i, \alpha} \cdot e_1 ^{\alpha_1} e_2^{\alpha_2}\cdots e_n^{\alpha_n}.$$
In other words, the $p$-relative embedding of $\widetilde{X}$ in $\mathbb{P}^{n-1} \times   \widetilde{Y}$:
$$
\xymatrix{
\widetilde{X} \ar[d]^p \ar[r] & \mathbb{P}^{n-1} \times   \widetilde{Y} \ar[d]^p\\
\widetilde{Y} \ar[r] & \widetilde{Y} }
$$
is defined by the polynomials $\varphi (s_i)$ whose coefficients are independent of $\widetilde{Y}$.
Then $p$ is locally trivial and we have the splitting
$\widetilde{X}\simeq \widetilde{Y} \times F$,
where $F$ is the generic fibre of $p$. 
\end{proof}

\medskip

In the second part of this section, we would like to 
recall some results about the positivity of direct images. For more details, 
we refer to \cite{BP08,BP10,Fuj16,Hor10,Kaw82,Kaw98,Kol85,PT14,Tsu10,Vie83,Vie95} to quote only a few.

\medskip

To be begin with, we first recall the definition of possibly singular hermitian metrics. We refer to \cite{Dem12} for more details.
Let $X$ be a projective manifold and let $L\rightarrow X$ be a line bundle on $X$ endowed
with a Hermitian metric $h_L$. We make the convention that, unless explicitly mentioned otherwise, the metrics in this article are allowed to be singular. 
Let $\Omega \subset X$ be any trivialization open set for $L$ and $e_L$ be a basis of $L$ over $\Omega$. Then
$$|e_L|_{h_L} ^2 = e^{-\varphi}$$ 
for some function $\varphi \in L^1 _{\loc} (\Omega)$. We say that $\varphi$ is the weight of $h_L$. 
Thanks to the Lelong-Poincar\'{e} formula, we know that
\begin{equation}\label{LP}
\frac{i}{\pi}\Theta_{h_L}(L) = dd^c \varphi . 
\end{equation}

We now recall the definition of the multiplier ideal sheaves cf. \cite[5.B]{Dem12} for more details.
Let $m\in \mathbb{N}$. Let $\mathcal{J} (h_L ^{\frac{1}{m}})\subset \mathcal{O}_X$ be the germs of holomorphic function $f\in \mathcal{O}_{X, x}$
such that $|f|^2 e^{-\frac{\varphi}{m}}$ is integrable near $x$. It is well known that $\mathcal{J} (h_L ^{\frac{1}{m}})$
is a coherent sheaf. If $\frac{i}{\pi}\Theta_{h_L}(L) \geq 0$ in the sense of current, thanks to \eqref{LP}, the weight $\varphi$ is a psh function.
Therefore, for $m\in\mathbb{N}$ large enough, we have $\mathcal{J} (h_L ^{\frac{1}{m}}) =\mathcal{O}_X$.

\medskip

The following result is a very special version of the standard Ohsawa-Takegoshi type extension theorem.
We refer to for example \cite{Man93, Dem12} among many others for the more general versions.

\begin{theorem}\label{OTver}\cite{Man93, Dem12}
Let $p : X\rightarrow Y$ be a fibration between two projective manifolds and let $L_Y$ 
be a very ample line bundle on $Y$ such that the global sections of $L_Y$ separates all $2 n$-jets,
where $n$ is the dimension of $Y$.
Let $L$ be a pseudo-effective line bundle on $X$ with a possibly singular hermitian metric $h$ such that $i\Theta_h (L) \geq 0$ on $X$.
Let $y\in Y$ be a generic point. Then the following restriction 
$$H^0 (X, \mathcal{O}_X (K_X + L+ p^\star L_Y ) \otimes \mathcal{J} (h)) \rightarrow H^0 (X_y, \mathcal{O}_{X_y} (K_X + L+ p^\star L_Y) \otimes \mathcal{J} (h |_{X_y})) $$
is surjective.
\end{theorem}

\medskip

We will use the following theorem to study the positivity of direct images in this article. 
It is a consequence of \cite{BP08, BP10, PT14}.

\begin{theorem}\cite{BP10,PT14}\label{extension}
Let $p: X \rightarrow Y$ be a fibration
between two projective manifolds and let $L$ be a pseudo-effective line bundle on $X$ with a possibly singular metric $h_L$ such that
$i \Theta_{h_L} (L) \geq 0$ in the sense of current. 
Let $m$ be a positive number such that $\mathcal{J} (h_L ^{\frac{1}{m}} |_{X_y}) =\mathcal{O}_{X_y}$
for a generic fibre $X_y$.
If $p_\star (m K_{X/Y} +L) \neq 0$, then $\det p_\star (m K_{X/Y} +L)$ is a pseudo-effective line bundle on $Y$.

\smallskip

Moreover, let $A_Y$ be a very ample line bundle on $Y$ such that the global sections of $A_Y -K_Y$ separates all $2 n$-jets,
where $n$ is the dimension of $Y$.
Then the restriction
\begin{equation}\label{surjlee}
H^0 (X, m K_{X/Y} + L  +p^\star A_Y) \rightarrow H^0 (X_y ,  m K_{X/Y} + L + p^\star A_Y) 
\end{equation}
is surjective for a generic $y \in Y$.
\end{theorem}

\begin{remark}
Note that the choice of $A_Y$ depends only on $Y$ and is independent of the fibration $p: X\rightarrow Y$, $L$ and $m$. 
This will be crucial in our article.
\end{remark}

\begin{proof}
We explain briefly the proof. Since $p_\star (m K_{X/Y} +L) \neq 0$, by \cite[A.2.1]{BP10}, there exists a $m$-relative Bergman type metric $h_{m,B}$ on $m K_{X/Y} + L$ 
with respect to $h_L$ such that $i\Theta_{h_{m, B}} (m K_{X/Y}+L) \geq 0$. 
Then $h:= \frac{m-1}{m}h_{m,B} + \frac{1}{m} h_L$ defines a possibly singular metric on
$$\widetilde{L} : = \frac{m-1}{m} (m K_{X/Y} + L) + \frac{1}{m} L ,$$
with $i\Theta_h (\widetilde{L} ) \geq 0$. By construction, we have
\begin{equation}\label{equality}
m K_{X/Y} + L  = K_{X/Y} + \widetilde{L}. 
\end{equation}

\medskip

Let $y\in Y$ be a generic point and let $\varphi_m$ be the weight
of $h_{m,B}$.
Then for every $s \in H^0 (X_y ,  m K_{X/Y} + L )$, by the construction of the $m$-relative Bergman kernel metric, 
$|s|^2 e^{-\varphi_m}$ is $C^0$-bounded. 
Combining this with the assumption $\mathcal{J} (h_L ^{\frac{1}{m}} |_{X_y}) =\mathcal{O}_{X_y}$, we know that 
\begin{equation}\label{normcon}
\int_{X_y} |s|_h ^2 < +\infty. 
\end{equation}
Therefore the inclusion
$$p_\star (\mathcal{O}_X (K_{X/Y}+ \widetilde{L}) \otimes \mathcal{J} (h)) \subset p_\star (K_{X/Y} +\widetilde{L})$$
is generically isomorphic. By applying \cite[Thm 3.3.5]{PT14} and \cite[Cor 2.9]{CP17}, 
we know that
$\det p_\star (K_{X/Y} +\widetilde{L})$ is pseudo-effective. Therefore $\det p_\star (m K_{X/Y} +L)$ is pseudo-effective.
\medskip

To prove the surjectivity of \eqref{surjlee}, we can first assume that 
$$H^0 (X_y ,  m K_{X/Y} + L + p^\star A_Y) \neq 0 \qquad\text{for a generic } y\in Y ,$$
which is equivalent to say that $p_\star (m K_{X/Y} +L) \neq 0$. Note that 
$$K_X + \widetilde{L} +p^\star (A_Y - K_Y) = m K_{X/Y} +L + p^\star (A_Y).$$
By applying Theorem \ref{OTver} to the line bundle $K_X + \widetilde{L} +p^\star (A_Y - K_Y)$, 
thanks to \eqref{normcon}, we know that $s$ be can extended to a section in $H^0 (X, m K_{X/Y} + L +p^\star A_Y)$.
In other words, the restriction \eqref{surjlee} is surjective. 
\end{proof}

We need two slight generalizations of the above theorem. The first is a direct consequence of Theorem \ref{extension} and the argument in 
\cite[Lemma 5.4]{CP17}.
\begin{proposition}\label{lowercontr}
Let $p: X \rightarrow Y$ be a fibration
between two projective manifolds and let $L$ be a line bundle on $X$ with a possibly singular metric $h_L$ such that
$i \Theta_{h_L} (L) \geq  p^\star \alpha$ in the sense of current for some smooth $d$-closed $(1,1)$-form $\alpha$ on $Y$. 
Let $m$ be a positive number such that $\mathcal{J} (h_L ^{\frac{1}{m}} |_{X_y}) =\mathcal{O}_{X_y}$
for a generic fibre $X_y$ and $p_\star (m K_{X/Y} +L) \neq 0$. 
Then $h_L$ induces a metric $\widehat{h}$ on $\det p_\star (m K_{X/Y} +L)$ 
such that
$$i\Theta_{\widehat{h}} (\det p_\star (m K_{X/Y} +L)) \geq r \cdot \alpha \qquad\text{ on } Y $$
in the sense of current, where $r=\rank p_\star (m K_{X/Y} +L)$.
\end{proposition}

\begin{proposition}\label{posdir}
Let $p: X \rightarrow Y$ be a fibration
between two projective manifolds and let $A_Y$ be a very ample line bundle on $Y$ in Theorem \ref{extension}.
Let $F$ be a pseudo-effective line bundle on $X$ with a possibly singular metric $h_F$ such that
$i\Theta_{h_F} (F) \geq 0$ in the sense of current
and let $m \in \mathbb{N}$ be a number such that 
$\mathcal{J} (h_F ^{\frac{1}{m}} |_{X_y}) =\mathcal{O}_{X_y}$ for a generic fibre $X_y$.

Let $Q \geq 0$ be some effective divisor on $X$ such that the support of $Q$ does not meet the general fibre of $p$.
Let $N$ be a line bundle such that $N+ \epsilon F + p^\star A_Y$ is semipositive
\footnote{It means that there exists a smooth hermitian metric such that the curvature is semipositive. In particular, if
$N+ \epsilon F + p^\star A_Y$ is $\mathbb{R}$-semiample, then it is semipositive.} 
for some $0<\epsilon <1$.
Then the restriction 
$$H^0 (X,  m K_{X/Y} + N +F + Q  + 2 p^\star A_Y) \rightarrow H^0 (X_y , m K_{X/Y} + N+ F +Q + 2 p^\star A_Y) $$
is surjective.
\end{proposition}

\begin{proof}
As $N + p^\star A_Y + \epsilon F$ is semipositive, it can be equipped with a smooth metric $h_1$ with semi-positive curvature.
Let $h_Q$ be a singular metric on $Q$ such that $i\Theta_{h_Q} (Q) =[Q]$.
Then $h=h_1 + (1-\epsilon) h_F +h_Q$ defines a metric on the line bundle 
$$N + F  + Q+   p^\star A_Y = ( N +\epsilon F + p^\star A_Y)+  (1-\epsilon) F +Q$$
with $i\Theta_h (N + F  + Q + p^\star A_Y) \geq 0$. Since  $\mathcal{J} (h_F ^{\frac{1}{m}} |_{X_y}) =\mathcal{O}_{X_y}$ and $p (Q) \subsetneq Y$,
we have $\mathcal{J} (h^{\frac{1}{m}} |_{X_y}) =\mathcal{O}_{X_y}$.
The pair $(L, h_L) := (N + F  +Q +  p^\star A_Y, h)$ satisfies thus the condition in Theorem \ref{extension}.
By applying Theorem \ref{extension}, we know that
$$H^0 (X , m K_{X/Y} +N +F +Q + 2 p^\star A_Y) \rightarrow H^0 (X_y , m K_{X/Y} + N +F +Q + 2 p^\star A_Y)$$
is surjective.
\end{proof}

Together with the arguments in \cite{Zha05}, we have
\begin{corollary}\label{usefulcor}
Let $p: X \rightarrow Y$ be a fibration
between two projective manifolds and let $A_Y$ be a very ample line bundle on $Y$ in Theorem \ref{extension}.
If $-K_{X/Y}$ is nef, then for every $p$-ample pseudo-effective line bundle $L$ on $X$, the restriction map
$$H^0 (X,  L  + 2 p^\star A_Y) \rightarrow H^0 (X_y , L + 2 p^\star A_Y)$$
is surjective for a generic $y \in Y$. 
\end{corollary}

\begin{proof}
Since $L$ is pseudo-effective, there exists a possibly singular metric $h_L$ such that $i\Theta_{h_L} (L) \geq 0$.
Let $m\in\mathbb{N}$ large enough such that $\mathcal{J} (h_L ^{\frac{1}{m}} |_{X_y}) =\mathcal{O}_{X_y}$ for a generic fibre $X_y$. As $L$ is $p$-ample, $\epsilon L + p^\star A_Y$ is ample for some $0 < \epsilon < 1$.
Combining this with the nefness of $-m K_{X/Y}$, we know that $-m K_{X/Y} + \epsilon L + p^\star A_Y$ is ample.
The corollary is thus proved by using Proposition \ref{posdir}, where we take $N= -m K_{X/Y}$, $F=L$ and $Q=\mathcal{O}_X$.
\end{proof}

\section{Two propositions}\label{twolemme}

Let $p: X \rightarrow Y$ be a fibration between two projective manifolds and let $L$ be an ample line bundle on $X$. 
If $p$ is a trivial fibration\footnote{It means that 
$X \simeq Y \times F$ where $F$ is a 
generic point of $p$.} and $p_\star (L)$ is non zero, we know that $\det p_\star (L)$ is ample and $L- \frac{1}{r} p^\star\det p_\star (L)$ is semi-ample 
where $r$ is the rank of $p_\star (L)$. The goal of this section is to prove some similar results when $p$ is smooth in codimension $1$ and $-K_{X/Y}$ is nef.

\medskip

To begin with, by combining the diagonal method of Viehweg \cite[Thm 6.24]{Vie95} with the method in \cite{Zha05}, 
we can prove the following key proposition of the article.
The proof is very close to \cite[Thm 6.24]{Vie95}\cite[Section 2.6]{Tsu10}\cite[Thm 3.13]{CP17}.

\begin{proposition}\label{psf}
Let $p: X\to Y$ be a fibration between two projective manifolds and we assume that 
$-K_{X/Y}$ is nef. 
We suppose that $p$ is smooth in codimension $1$, namely $p$ is smooth outside a subvariety $Z$ in $X$
of codimension at least $2$.
Let $A$ be a $p$-ample line bundle on $X$ such that $p_\star (A)$ is locally free.
Then $r A - p^\star\det p_\star (A)$ is pseudo-effective, where $r$ is the rank of $p_\star (A)$.
\end{proposition}

\begin{proof}[Sketch of the proof]
To explain the idea of the proof, 
we first sketch the proof under the assumption that $p$ is smooth. 

We consider the natural morphism
$$s: \det p_\star (A) \rightarrow \bigotimes^r p_\star (A) .$$
Let $X^{r} = X\times_Y X \times_Y \cdots \times_Y X$ be the $r$-times fiberwise product of the fibration $p: X\rightarrow Y$.
Let $\pr_i: X^{r} \rightarrow X$ be the $i$-th directional projection and let $\p : X^r \rightarrow Y$ be the natural induced fibration.
Set $A_r := \bigotimes\limits_{i=1}^r \pr_i ^\star A$ and $L:= A_r - (\p)^\star\det p_\star (A)$. 
As $(\p)_\star (A_r)= \bigotimes\limits^r p_\star (A) $, the morphism $s$ induces a non-trivial section 
\begin{equation}\label{section}
 \tau \in H^0 (X^{r} , L) .
\end{equation}

\smallskip

The idea of Viehweg is as follows. 
Let $j: X \rightarrow X^r$ be the diagonal embedding.
We know that 
\begin{equation}\label{diagg}
L |_{j (X)} = r A - p^\star\det p_\star (A) . 
\end{equation}
Note that \eqref{section} implies that $L$ is effective on $X^r$. 
If the effectiveness of $L$ on $X^r$ implies the pseudo-effectiveness of $L |_{j (X)}$, thanks to \eqref{diagg}, the proposition is proved.
In general, it is not true. However,
by using Corollary \ref{usefulcor}, we can prove it in our case.

\smallskip

To be more precise, let $A_Y$ be the ample line bundle on $Y$ in Theorem \ref{extension}.
For every $q \in \mathbb{N}^\star$, $q L$ is effective and $\p$-ample.
Since $-K_{X^r /Y} = \sum_i \pr_i ^\star (-K_{X/Y})$ is nef, we can apply Corollary \ref{usefulcor} to $(\p: X^r \rightarrow Y,  q L)$.
Then the restriction 
\begin{equation}\label{sur0}
 H^0 (X^{r}, q L +2 (\p)^\star A_Y) \rightarrow H^0 (X^r _y,  q L +2 (\p)^\star A_Y) 
\end{equation}
is surjective for a generic fibre $X^r _y$. 

Let $j: X\rightarrow X^r$ be the diagonal embedding. By restricting \eqref{sur0} to $j (X)$, thanks to \eqref{diagg}, we have 
$$
\xymatrix{
H^0 \big(X^r , q L +2 (\p)^\star A_Y)  \ar[d]_{j^\star} \ar@{->>}[r] & H^0 \big(X^r _y , q L +2 (\p)^\star A_Y) \ar[d]_{j^\star} \\
H^0 (X, qr A - q \cdot p^\star \det p_\star (A) +2 p^\star A_Y ) \ar[r]
& H^0 (X_y ,qr A )}
$$
Since the image of $H^0 \big(X^r _y , q L +2 (\p)^\star A_Y) =H^0 (X^r _y, q A_r)  \rightarrow H^0 (X_y ,qr A )$ is non trivial, the surjectivity \eqref{sur0}
and the above commutative diagram implies that the image of
$$H^0 (X, qr A - q \cdot p^\star \det p_\star (A) +2 p^\star A_Y) \rightarrow H^0 (X_y, qr A) $$
is non zero.  In particular, $qr A - q \cdot p^\star \det p_\star (A) +2 p^\star A_Y$ is effective on $X$.
Then $rA - p^\star \det p_\star (A) +\frac{2}{q} \cdot  p^\star A_Y$ is $\mathbb{Q}$-effective for every $q >0$.
By letting $q\rightarrow +\infty$, the proposition is proved.
\end{proof}

Now we give the complete proof of the proposition, which follows closely \cite[Thm 3.13]{CP17}.

\begin{proof}[Proof of the proposition \ref{psf}]
Let $Y_1$ be the flat locus of $p$. As $p$ is smooth in codimension $1$, $p^{-1} (Y \setminus Y_1)$ is of codimension at least two.
After replacing $Z$ by $Z \cup p^{-1} (Y \setminus Y_1)$, we can assume that 
$p$ is flat over $p (X\setminus Z)$ and $Z$ is of still codimension at least $2$.
Let $X^{r} = X\times_Y X \times \cdots \times_Y X$ be the $r$-times fiberwise product of $p: X\rightarrow Y$, and let $X^{(r)}$ be a desingularisation of $X^r$. 
Let 
$$\pr_i: X^{(r)} \rightarrow X$$ 
be the $i$-th directional projection, and $\p : X^{(r)} \rightarrow Y$ be the natural induced morphism.
Set $E:= X^{(r)} \setminus (\bigcap\limits_i \pr_i ^{-1} (X\setminus Z))$.
Then the diagonal embedding $j: X\setminus Z  \hookrightarrow  X^{(r)}$ satisfies $j (X\setminus Z) \subset X^{(r)}\setminus E$. 
Note that for a generic point $y\in Y$, as $p$ is smooth over $y$, 
we have 
\begin{equation}\label{smoofib}
X^{(r)} _y \cap E =\emptyset , 
\end{equation}
where $X^{(r)} _y$ is the fibre over $y$.

\smallskip

We consider the natural morphism
$$s: \det p_\star (A) \rightarrow \bigotimes^r p_\star (A) .$$
Set $A_r := \bigotimes\limits_{i=1}^r \pr_i ^\star A$ and $L:= A_r - (\p)^\star\det p_\star (A)$. 
Since $p$ is flat over $p (X\setminus Z)$ and $p_\star (A)$ is locally free, 
thanks to \cite[Lemma 3.15]{Hor10}, $s$ induces a non-trivial section $\tau \in H^0 (X^{(r)} , L +E')$ 
for some divisor $E'$ supported in $E$.
Let $A_Y$ be the ample line bundle on $Y$ in Theorem \ref{extension}. 
\medskip

For every $q \geq 1$, we first prove that there exists a divisor $E_q$ supported in $E$ such that the restriction
\begin{equation}\label{sur}
 H^0 (X^{(r)}, q L + E_q + 2 (\p)^\star A_Y) \rightarrow H^0 (X^{(r)} _y,  q L +2 (\p)^\star A_Y) 
\end{equation}
is surjective for a generic $y \in Y$. 

\medskip

In fact, there exist some effective divisors $E_1$ and $E_2$ supported in $E$ such that
\begin{equation}\label{add2}
- K_{X^{(r)}/Y} = \sum_{i=1}^r \pr_i ^\star (-K_{X/Y}) + E_1 -E_2 . 
\end{equation}
Then for every $m \in\mathbb{N}$, we have
\begin{equation}\label{add1}
q L + qm E_2 + 2 (\p)^\star A_Y = qm K_{X^{(r)} /Y} + qm (\sum_{i=1}^r \pr_i ^\star (-K_{X/Y})) +q L + qm E_1 +2 (\p)^\star A_Y . 
\end{equation}
As $-K_{X/Y}$ is nef, for $\epsilon \ll 1$, the line bundle
$$qm (\sum_i \pr_i ^\star (-K_{X/Y})) + \epsilon L + (\p)^\star A_Y = \sum_i \pr_i ^\star( - qm K_{X/Y} +\epsilon A) +  (\p)^\star (A_Y - \ep\det p_\star (A))$$
is semi-ample on $X^{(r)}$. 
 
We can thus apply Proposition \ref{posdir} to the fibration $\p : X^{(r)}\rightarrow Y$ by taking $N =qm (\sum_{i=1}^r \pr_i ^\star (-K_{X/Y}))$,
$F= q L$ and $Q=qm E_1$, where $m \gg 1$ is large enough with respect to $q L$. Together with \eqref{add1}, the restriction
\begin{equation}\label{add3}
H^0 (X^{(r)}, q L + qm E_2 + 2 (\p)^\star A_Y) \rightarrow H^0 (X^{(r)} _y, q L + qm E_2 + 2 (\p)^\star A_Y ) 
\end{equation}
is thus surjective, where $y\in Y$ is a generic point. 
Thanks to \eqref{smoofib}, we have  
$$H^0 (X^{(r)} _y, q L + qm E_2 + 2 (\p)^\star A_Y ) = H^0 (X^{(r)} _y, q L + 2 (\p)^\star A_Y ) .$$
Combining this with \eqref{add3}, \eqref{sur} is proved by taking $E_q = qm E_2$. 
\medskip

Finally, we take the pull-back $j^\star$ of \eqref{sur}, where $j: X\setminus Z \rightarrow X^{(r)}$ is the diagonal embedding.
As $j( X\setminus Z) \subset X^{(r)}\setminus E$ and $E_q$ is supported in $E$, we obtain the following commutative diagram
$$
\xymatrix{
H^0 \big(X^{(r)} , q L  + 2 (\p)^\star A_Y + E_q )  \ar[d]_{j^\star} \ar@{->>}[r] & H^0 \big(X^{(r)} _y , q L + 2 (\p)^\star A_Y) \ar[d]_{j^\star}\\
H^0 (X\setminus Z, qr A - q \cdot p^\star \det p_\star (A) +2 p^\star A_Y ) \ar[r]
& H^0 (X_y ,qr A )} .
$$
Note that the image of $ H^0 \big(X^{(r)} _y , q L + 2 (\p)^\star A_Y) \rightarrow H^0 (X_y ,qr A )$ is non zero. Together with the surjectivity \eqref{sur}
and the above commutative diagram, we know that the image of 
$$H^0 (X\setminus Z, qr A - q \cdot p^\star \det p_\star (A) +2 p^\star A_Y) \rightarrow H^0 (X_y, qr A)$$
is non zero. As $\codim_X Z \geq 2$, $qr A - q \cdot p^\star \det p_\star (A) +2 \cdot p^\star A_Y$ is effective. 
Then $rA - p^\star \det p_\star (A) +\frac{2}{q} \cdot p^\star A_Y$ is $\mathbb{Q}$-effective.
The proposition is proved by letting $q\rightarrow +\infty$.
\end{proof}

Using \cite{Zha05,PT14}, we can prove

\begin{proposition}\label{addedlemma}
Let $p: X \rightarrow Y$ be a fibration between two projective manifolds and we suppose that $-K_{X/Y}$ is nef. 
Let $L$ be a pseudo-effective and $p$-ample line bundle on $X$. 
If $ p_\star (L)$ is not zero, then $\det p_\star (L)$ is pseudo-effective.
\end{proposition}

\begin{proof}
Let $h_L$ be a possibly singular metric on $L$
such that $i\Theta_{h_L} (L) \geq 0$ on $X$ and let $m\in \mathbb{N}$ large enough such that $\mathcal{J} (h_L ^{\frac{1}{m}}) = \mathcal{O}_X$.
Since $L$ is $p$-ample, there exists an ample line bundle $A_Y$ on $Y$ 
such that $L + p^\star A_Y$ is ample on $X$.
Then for every $\ep >0$, $- mK_{X/Y} + \ep (L +  p^\star A_Y)$ is ample.
Therefore we can find a smooth metric $h_\ep$ on $- mK_{X/Y} + \ep L$
such that
$$i\Theta_{h_\ep} (- mK_{X/Y} + \ep L) \geq -\ep p^\star \omega_Y ,$$
where $\omega_Y$ is a $(1,1)$-form in the class of $c_1 (A_Y)$.

Then $\widetilde{h}_\ep := h_\ep + (1-\ep) h_L$ defines a metric on $- mK_{X/Y} + L$
such that 
$$i\Theta_{\widetilde{h}_\ep} (- mK_{X/Y} +  L) \geq -\ep p^\star \omega_Y \qquad\text{and}\qquad \mathcal{J} (\widetilde{h}_\ep ^\frac{1}{m}) =\mathcal{O}_X .$$
By applying Proposition \ref{lowercontr}, $\widetilde{h}_\ep$ induces a metric $\widehat{h}_\ep$ on 
$$\det p_\star (m K_{X/Y} + (- mK_{X/Y} +  L)) $$ 
such that 
$$i\Theta_{\widehat{h}_\ep} (\det p_\star (m K_{X/Y} + (- mK_{X/Y} +  L)) )\geq -\ep \rank p_\star (L) \cdot p^\star\omega_Y .$$
The proposition is proved by letting $\ep\rightarrow 0$.
\end{proof}

\section{Proof of the main theorem}\label{mainsect}

We begin to prove the main theorem :

\begin{theorem}\label{mainproof}
Let $X$ be a projective manifold with nef anti-canonical bundle and let $p: X\to Y$ be the Albanese map. 
Then $p$ is locally trivial.
Moreover, let $\mathbb{C}^r \rightarrow Y$ be the universal cover and set $\widetilde{X} := X\times_Y \mathbb{C}^r$. Then
$\widetilde{X}$ admits the following splitting 
$$\widetilde{X} \simeq \mathbb{C}^r \times F ,$$
where $F$ is the generic fibre of $p$.
\end{theorem}

\begin{proof}

First of all, thanks to \cite{LTZZ10}, $p$ is flat.
We can thus find a very ample line bundle $A$ on $X$ such that $p_\star (m A)$ is locally free for every $m\in\mathbb{N}$
and the natural morphism $\Sym^m p_\star (A) \rightarrow p_\star (m A)$ is surjective for every $m\in\mathbb{N}$.
Let $r$ be the rank of $p_\star (A)$.
After passing to some isogeny of $Y$, we can assume that $\frac{1}{r} \det p_\star (A)$ is a line bundle. 
Set
\begin{equation}\label{def}
L := A - \frac{1}{r} p^\star \det p_\star (A) . 
\end{equation}
Thanks to \cite{LTZZ10}, $p$ is smooth in codimension $1$.
We can thus apply Proposition \ref{psf} to $(p: X\rightarrow Y, A)$. Therefore $L$ is pseudo-effective, and by construction we have
\begin{equation}\label{numtrivial}
c_1 (p_\star (L))=0 .
\end{equation}

\medskip

The plan of the rest of the proof is as follows. 
From Step 1 to Step 3,  by combining \cite[3.D]{DPS94} with the results about the positivity of direct images, 
we will prove that $p_\star (L) $ is numerically flat on $Y$. In Step 4, we will prove the theorem.
We remark that, if $-K_X$ is hermitian positive, we can easily prove that $p_\star (L)$ is hermitian flat by using \cite[Thm 3.3.5]{PT14}, \cite[Thm 5.2]{CP17}
and the arguments in \eqref{keytrick}.
However, as $-K_X$ is only nef in our case, we can not use directly \cite{PT14}. 
We use here the isogeny argument \cite[Lemma 3.21]{DPS94} to prove the nefness of $p_\star (L)$.

\medskip

{\em Step 1: Construction}

Let $A_Y$ be a sufficiently ample line bundle on $Y$ such that $A_Y - \frac{1}{r}  \det p_\star (A)$ is ample, 
and $A_Y$ satisfies the condition in Theorem \ref{extension}. We can ask also that, for every $F\in \Pic^0 (Y)$, $A_Y +F$ is very ample on $Y$.

Since $Y$ is a torus, for every $n\in\mathbb{N}$ sufficiently divisible, 
we can take a $n$ to $1$ isogeny $\pi_{Y,n}:  Y \rightarrow Y$. Let $X_n := X\times_{\pi_{Y,n}} Y$. Then $X_n$ is smooth with nef anticanonical bundle,
and we have 
\begin{equation}\label{isoge}
\pi_{Y,n} ^\star  c_1 (A_Y )= n \cdot c_1 (A_Y) \in H^{1,1} (Y, \mathbb{R}) . 
\end{equation}
$$
\xymatrix{
X_n \ar[d]_{p_n} \ar[r]^{\pi_n} & X \ar[d]^p\\
Y \ar[r]^{\pi_{Y,n}} & Y}
$$
Set $V_n :=  \pi_{Y,n} ^\star p_\star (L)$. As $L$ is proved to be pseudo-effective, by applying Corollary \ref{usefulcor} to $( p_n: X_n \rightarrow Y, \pi_n ^\star L)$, 
the restriction
\begin{equation}\label{genericsur}
H^0 (Y, 2 A_Y \otimes V_n) \rightarrow (2 A_Y \otimes V_n )_y  
\end{equation}
is surjective for a generic $y\in Y$.

\medskip

{\em Step 2: Global surjectivity}

We prove in this step that for $n$ sufficiently large and divisible, the restriction 
\begin{equation}\label{surjeverywhere}
H^0 (Y, 3 A_Y \otimes V_n ) \rightarrow (3 A_Y \otimes V_n)_{y_0}  
\end{equation}
is surjective for every $y_0 \in Y$.

\medskip

In fact, thanks to the generic surjectivity \eqref{genericsur}, we can find 
$$\{s_1, \cdots ,s_r \} \subset H^0 (Y, 2 A_Y \otimes V_n)$$ 
such that $s:= s_1 \wedge \cdots \wedge s_r \in H^0 (Y, 2r A_Y \otimes \Det V_n)$
is non zero. By \eqref{numtrivial}, we know that
$$c_1 (\Det V_n)= \pi_{Y,n} ^\star c_1 (\Det p_\star (L))=0 .$$
Therefore the numerical class
\begin{equation}\label{indepent}
c_1 (\Div s) = c_1 (2r A_Y) \in H^{1,1} (Y, \mathbb{R}) 
\end{equation}
is independent of $n$. 

\medskip

On the other hand, after a translation, we can suppose without lose of generality that $\pi_{Y, n} (y_0)$ is the origin in $Y$.
Then $\{\pi_{Y, n} ^{-1} (\pi_{Y, n} (y_0))\}$ is the set of the $n$-torsion points in $Y$. 
Thanks to \eqref{indepent}, the numerical class $c_1 (\Div (s))$ is independent of $n$. As a consequence, 
\cite[Prop 7.7]{MFK} implies that for $n$ large enough, 
the divisor $\Div (s)$ could not contain the set $\{\pi_{Y, n} ^{-1} (\pi_{Y, n} (y_0))\}$
\footnote{The proof in \cite[Prop 7.7]{MFK} is an effective estimate.  We can also give non effective estimate proof as follows. 
In fact, if it is not true, then for every $n$ sufficiently divisible, 
we can find an ample line bundle in the same class of $2 r c_1 (A_Y)$, such that the set 
$\{\pi_{Y, n} ^{-1} (\pi_{Y, n} (y_0))\}$ is contained in the zero locus $Z_n$ of a
section of this line bundle. As the volumes of $[Z_n]$ is independent of $n$, by the compactness of cycle spaces, after passing to a subsequence,
$[Z_n]$ will tends to a divisor $[Z]$ in $Y$. However, the torsion sets $\{\pi_{Y, n} ^{-1} (\pi_{Y, n} (y_0))\}$ will not converge to a strict subvariety of $X$ when $n \rightarrow +\infty$.
We get thus a contradiction.}.
Therefore there exists a point $y_1 \in \pi_{Y, n} ^{-1} (\pi_{Y, n} (y_0)) $ such that
\begin{equation}\label{nonvan}
s (y_1) \neq 0 .
\end{equation}
\smallskip

Finally, let $G_n$ be the Galois group associated to the \'{e}tale cover $\pi_{Y, n}: Y \rightarrow Y$ and 
let $g\in G_n$ such that $g(y_0)=y_1$. As $V_n$ is $G_n$-invariant, $g$ induces the isomorphisms
$$H^0 (Y,  2 A_Y \otimes V_n) \rightarrow H^0 (Y,  2 g^\star (A_Y) \otimes V_n) $$ 
and 
$$H^0 (Y,  2 r A_Y \otimes \det V_n) \rightarrow H^0 (Y,  2r g^\star (A_Y) \otimes \det V_n) .$$
Therefore
$$ g^\star (s_1) (y_0)\wedge  \cdots \wedge g^\star (s_r) (y_0) =g^\star (s) (y_0) =s (y_1) \neq 0 .$$
As a consequence, $\{g^\star (s_1) (y_0), \cdots , g^\star (s_r) (y_0)\}$ generates $(2 g^\star (A_Y) \otimes V_n)_{y_0} $. 
Note that $A_Y - g^\star A_Y \in \Pic^0 (Y)$. The construction of $A_Y$ implies thus 
that 
$$A_Y + 2 ( A_Y - g^\star A_Y)$$ is very ample. Therefore we can find a section $\tau \in H^0 (Y, 3 A_Y - 2 g^\star A_Y)$
such that $\tau (y_0) \neq 0$. Then
$\{\tau \otimes g^\star (s_1) (y_0), \cdots , \tau \otimes g^\star (s_r) (y_0)\}$ generates $(3 A_Y \otimes V_n)_{y_0} $
and \eqref{surjeverywhere} is proved.

\bigskip

{\em Step 3: Numerically flatness of $p_\star (L)$}

Let $\mathbb{P} ( p_\star (L))$ (resp. $\mathbb{P} ( V_n ) $) be the projectivization of $p_\star (L)$ (resp. $V_n$). 
We have the commutative diagram
$$
\xymatrix{
\mathbb{P} ( V_n ) \ar[d]_{p_n} \ar[r]^{\pi_n} & \mathbb{P} ( p_\star (L)) \ar[d]^p\\
Y \ar[r]^{\pi_{Y,n}} & Y}
$$
Let $\omega_Y$ be a K\"{a}hler metric in the same class of $A_Y$.
Thanks to \eqref{surjeverywhere}, we can find a smooth metric $h$ on $\mathcal{O}_{\mathbb{P} ( V_n)} (1)$ such that
$$i\Theta_{h} (\mathcal{O}_{\mathbb{P} ( V_n)} (1)) \geq - 3  p_n ^\star \omega_Y =
-\frac{3}{n} (\pi_n \circ p)^\star \omega_Y .$$
Note that $\pi_n ^\star \mathcal{O}_{\mathbb{P} ( p_\star (L))} (1) = \mathcal{O}_{\mathbb{P} ( V_n)} (1)$. Then 
$h$ induces a smooth metric $h_n$ on $\mathcal{O}_{\mathbb{P} ( p_\star (L))} (1)$ by taking the average of the translates of $h$
by the $\pi_{Y,n}$-torsion points. We have
$$i\Theta_{h_n} (\mathcal{O}_{\mathbb{P} ( p_\star (L))} (1)) \geq -\frac{3}{n} p^\star \omega_Y .$$
As this holds for every $n$ sufficiently large and divisible, $\mathcal{O}_{\mathbb{P} ( p_\star (L))} (1)$ is nef by definition. 
Therefore the vector bundle $p_\star (L)$ is nef.
Combining this with \eqref{numtrivial}, $p_\star (L)$ is thus numerically flat.

\bigskip

{\em Step 4: Final conclusion} 

Let $V := p_\star (L)$. As $A$ is very ample, $V$ induces a $p$-relative embedding
$$
\xymatrix{
X \ar[rd]_p \ar@{^{(}->}[rr]^j && \mathbb{P} (V) \ar[ld]^f\\
& Y}
$$
We have $L = j^\star \mathcal{O}_{\mathbb{P} ( V)} (1)$. 
For $m$ large enough, we have thus the exact sequence
\begin{equation}\label{exactseq}
0 \rightarrow f_\star (\mathcal{I}_X \otimes \mathcal{O} _{\mathbb{P} (V)} (m)) \rightarrow  
f_\star ( \mathcal{O} _{\mathbb{P} (V)} (m)) \rightarrow p_\star (m L) \rightarrow 0 . 
\end{equation}
Thanks to Step 4, $f_\star ( \mathcal{O} _{\mathbb{P} (V)} (m)) =\Sym^m V$ is numerically flat. 

\smallskip

{\em Claim :} $p_\star (m L)$ is numerically flat for every $m \geq 1$. 

\smallskip

We will postpone the proof of the claim to Lemma \ref{addedle} 
and first finish the proof of the theorem. As
$p_\star (m L)$ is numerically flat for every $m \geq 1$, by using Proposition \ref{isotri}, 
the theorem is proved.
\end{proof}

To complete the proof of the main theorem, it remains to prove the claim

\begin{lemma}\label{addedle}
The vector bundle $p_\star (m L)$ is numerically flat for every $m \geq 1$.
\end{lemma}

\begin{proof}
As $\Sym^m V$ is numerically flat, the exact sequence \eqref{exactseq} implies that 
$p_\star (m L)$ is nef. It remains to prove that $c_1 (p_\star (m L) ) =0$. 

\medskip

In fact, as $m L$ is $p$-ample and $p_\star (m L)$ is locally free,
Proposition \ref{psf} implies that 
$m L - \frac{1}{r_m} p^\star c_1 (p_\star (m L))$
is $\mathbb{Q}$-pseudo-effective, where $r_m$ is the rank of $p_\star (m L)$.
Then 
\begin{equation}\label{equL}
\widetilde{L} := L- \frac{1}{m \cdot r_m} p^\star c_1 (p_\star (m L))  
\end{equation}
is $\mathbb{Q}$-pseudo-effective.
After passing to some isogeny of $Y$, we can assume that $\frac{1}{m \cdot r_m} c_1 (p_\star (m L))$ is a line bundle. 
Therefore $\widetilde{L}$ is a pseudo-effective line bundle.
By taking the determinant of the direct image of \eqref{equL}, we get 
$$c_1 (\det p_\star  (\widetilde{L})) =   c_1 (\det p_\star (L)) - \frac{1}{m \cdot r_m} c_1 (p_\star (m L)) .$$
Combining this with \eqref{numtrivial}, we have
\begin{equation}\label{finaleq}
\frac{1}{m \cdot r_m} c_1 (p_\star (m L)) + c_1 (\det p_\star  (\widetilde{L})) =0\in H^{1,1} (Y, \mathbb{R}) . 
\end{equation}

To conclude, as $p_\star (m L)$ is proved to be nef, $c_1 (p_\star (m L))$ is pseudo-effective.
By construction, $\widetilde{L}$ is pseudo-effective and $p$-ample. Then Proposition \ref{addedlemma} implies that 
$\det p_\star (\widetilde{L})$ is also pseudo-effective.
Therefore \eqref{finaleq} implies that $c_1 (p_\star (m L))=0$ and the lemma is proved.
\end{proof}

\bigskip

We now discuss the structure of the universal cover of $X$. Let $X$ be a compact K\"ahler manifold with nef anticanonical bundle.
Thanks to \cite{Pau97, Pau12}, we know

\begin{proposition}\cite[Thm 2]{Pau97}\cite{Pau12}\label{isomalb}
Let $X$ be a compact K\"ahler manifold with nef anticanonical bundle. Then after a finite \'{e}tale cover of $X$, the Albanese map
$$p: X\rightarrow Y$$ 
induces an isomorphism of fundamental groups.
\end{proposition}

\begin{proof}
For readers' convenience, we recall briefly the main steps of the proof of \cite[Thm 2]{Pau97}. 
First of all, thanks to \cite{Zha96, Pau12}, the albanese map is surjective with connected fibers.
Let $(G_n)$ be the descending central series of $\pi_1 (X)$, i.e., $G_1 = \pi_1 (X)$, $G_{n+1} = [G, G_n]$.
Set $G'_n =\sqrt{G_n}$. 
By applying \cite[Thm 2.2]{Cam95} to $p$, as $p$ is a fibration and $\pi_1 (Y)$ is abelian,  we know that
$$p_\star : \pi_1 (X) / G'_n  \rightarrow \pi_1 (Y)$$
is an isomorphism for all $n$. 

\cite[Thm 1]{Pau97} shows that $\pi_1 (X)$ 
is virtually nilpotent. Therefore, up to a finite \'{e}tale cover of $X$, we can assume that $\pi_1 (X)$ is nilpotent and torsion free.
Then $\pi_1 (X) / G'_n = \pi_1 (X)$ for some $n \in \mathbb{N}$. As a consequence, 
$$p_\star : \pi_1 (X) \rightarrow \pi_1 (Y)$$
is an isomorphism.
\end{proof}

As an application, we have the following result.

\begin{corollary}
Let $X$ be a projective manifold with nef anticanonical bundle.
Then the universal cover $\widetilde{X}$ of $X$ admits the following splitting 
$$\widetilde{X} \simeq \mathbb{C}^r \times F .$$
Here $F$ is a simply connected projective manifold with nef anticanonical bundle and $r = \sup h^{1,0} (\widehat{X})$ where the supremum is taken over
all finite \'{e}tale covers $\widehat{X} \rightarrow X$.
\end{corollary}

\begin{proof}
After some finite \'{e}tale cover of $X$, we can assume that $h^{1,0} (X) =r$. Thanks to Proposition \ref{isomalb},
we can assume moreover that the Albanese map $p: X \rightarrow Y$ induces an isomorphism of fundamental groups :
\begin{equation}\label{isofund}
p_\star : \text{ } \pi_1 (X) \rightarrow \pi_1 (Y) . 
\end{equation}
Let $F$ be the generic fibre of $p$. Then $\dim Y =r$. 
Let $\pi: \mathbb{C}^r \rightarrow Y$ be the universal cover and set $X_1 := X\times_Y  \mathbb{C}^r$.
Theorem \ref{mainproof} implies that $p$ is locally trivial and we have the splitting 
\begin{equation}\label{coverX}
X_1 \simeq \mathbb{C}^r \times F .
\end{equation} 

It remains to prove that $F$ is simply connected. As $p$ is a submersion, we have the exact sequence
$$\pi_2 (Y) \rightarrow \pi_1 (F) \rightarrow \pi_1 (X) \rightarrow \pi_1 (Y) \rightarrow 1 .$$
Since $Y$ is a torus, we know that $\pi_2 (Y) =1$. Then the isomorphism \eqref{isofund} and the above exact sequence imply that
$\pi_1 (F)$ is trivial. The corollary is proved.
\end{proof}

\end{document}